\documentclass[12pt]{amsart}

\usepackage{amssymb}
\usepackage{amsthm}
\usepackage{amsmath}
\usepackage{fullpage}
\usepackage{ytableau}

\title[]{Computational Bounds for Doing Harmonic Analysis on \\ Permutation Modules of Finite Groups}
\subjclass[2010]{Primary 43A85, 20C40, 65T50}
\date{October 8, 2019}

\author[]{Michael Hansen}
\address{Department of Mathematics, Harvey Mudd College, Claremont, CA, USA}
\curraddr{}
\email{mhansen@gmail.com}
\thanks{}

\author[]{Masanori Koyama}
\address{Department of Mathematics, Harvey Mudd College, Claremont, CA, USA}
\curraddr{}
\email{koyama.masanori@gmail.com}
\thanks{}

\author[]{Matthew B.~A.~McDermott}
\address{Computer Science \& Artificial Intelligence Laboratory, Massachusetts Institute of Technology, Cambridge, MA, USA}
\curraddr{}
\email{mmd@mit.edu}
\thanks{}

\author[]{Michael E.~Orrison}
\address{Department of Mathematics, Harvey Mudd College, Claremont, CA, USA}
\curraddr{}
\email{orrison@hmc.edu}
\thanks{}

\author[]{Sarah Wolff}
\address{Department of Mathematics and Computer Science, Denison University, Granville, OH, USA}
\curraddr{}
\email{wolffs@denison.edu}
\thanks{}

\begin{document}

\parindent=0in
\parskip=.2in
\maketitle

\newtheorem{theorem}{Theorem}
\newtheorem{lemma}[theorem]{Lemma}
\newtheorem{proposition}[theorem]{Proposition}
\newtheorem{corollary}[theorem]{Corollary}
\newtheorem{definition}[theorem]{Definition} 
\newtheorem{example}[theorem]{Example} 
\newtheorem{note}[theorem]{Note}
\newtheorem*{remark}{Remark}


\begin{abstract}
We develop an approach to finding upper bounds for the number of arithmetic operations necessary for doing harmonic analysis on permutation modules of finite groups. The approach takes advantage of the intrinsic orbital structure of permutation modules, and it uses the multiplicities of irreducible submodules within individual orbital spaces to express the resulting computational bounds. We conclude by showing that these bounds are surprisingly small when dealing with certain permutation modules arising from the action of the symmetric group on tabloids. 
\end{abstract}

\section{Introduction}\label{Section:  Introduction}

Let $G$ be a finite group acting transitively on a finite set $X$.  Let $\mathbb{C}G$ be the complex group algebra of $G$, and let $\mathbb{C}X$ be the vector space of complex-valued functions defined on $X$.  The action of $G$ on $X$ turns $\mathbb{C}X$ into a $\mathbb{C}G$-permutation module, where if  $\alpha = \sum_{g \in G} \alpha(g) g \in \mathbb{C}G$ and $f \in \mathbb{C}X$, then
\[
(\alpha \cdot f)(x) = \sum_{g \in G} \alpha(g) f(g^{-1} x)
\]
 for all $x \in X$.  In this setting, we say that a basis $\mathcal{B}$ for $\mathbb{C}X$ is a \emph{harmonic basis} (with respect to $\mathbb{C}G$) if it can be partitioned into subsets such that each subset forms a basis for an irreducible $\mathbb{C}G$-submodule of $\mathbb{C}X$.

The problem addressed in this paper is the following:  Given an arbitrary $f \in \mathbb{C}X$, how may we efficiently compute the coefficients necessary to express $f$ in terms of a harmonic basis for $\mathbb{C}X$?  In other words, if $X = \{x_1, \dots, x_m \}$, and we are given $f(x_1), \dots, f(x_m)$, then how may we efficiently find the complex coefficients $\beta_1, \dots, \beta_m$ such that $f = \beta_1 b_1 + \cdots + \beta_m b_m$ for some harmonic basis $\mathcal{B} = \{ b_1, \dots, b_m \}$ of $\mathbb{C}X$?

The problem of computing such coefficients arises when doing \emph{harmonic analysis on finite groups}, especially when the function $f \in \mathbb{C}X$ corresponds to complex-valued data defined on a set $X$ with an underlying symmetry group $G$ (see, for example, \cite{ceccherini-silberstein-etal-2008, fassler-and-stiefel-1992, terras-1999}).  It also arises in \emph{generalized spectral analysis}, which was developed by Diaconis \cite{diaconis-1988, diaconis-1989} and which extends the classical spectral analysis of time series to the analysis of functions in $\mathbb{C}X$.  

As an example, let $X = \{ x_1, \dots, x_m \}$ and let $G$ be the cyclic group $\mathbb{Z}/m\mathbb{Z}$ acting on $X$ by cyclically permuting its elements.  The elements of $\mathbb{C}X$ may be viewed as discretized periodic \emph{signals} on $m$ points, in which case the irreducible submodules of $\mathbb{C}X$ correspond to different \emph{frequencies} in the usual signal processing sense.  In this case, the coefficients $\beta_1, \dots, \beta_m$ for $f \in \mathbb{C}X$ are then just fixed scalar multiples of the usual Fourier coefficients of $f$, which may be found by applying the usual discrete Fourier transform (DFT) to $f$.  The classical fast Fourier transform (FFT) may therefore be used to efficiently compute the $\beta_i$ using $O ( m \log m)$ arithmetic operations (see, for example, \cite{clausen-and-baum-1993, van-loan-1992}).  

As another example, suppose respondents in a survey are asked to choose their top $k$ favorite items from a set of $n$ items, where $k \le n/2$.  In this case, the set $X$ is the set of all $k$-element subsets of the items, and the group $G$ is the symmetric group $S_n$, whose natural action on the items induces an action on the set $X$.  If $f \in \mathbb{C}X$ is the function defined by setting $f(x)$ to be the number of people who choose the $k$-element set $x$, then the irreducible $\mathbb{C}S_n$-modules in $\mathbb{C}X$ correspond to summary statistics about $j$-element subsets of the items where $0 \le j \le k$.  As such, the associated $\beta_i$ for such survey data can be used to uncover hidden relationships among the $n$ items being ranked (see, for example, \cite{crisman-orrison-2017, diaconis-1988, diaconis-1989, marden-1995}).  

The techniques and insights for the analysis of such ``top $k$" survey data extend to the situation in which survey respondents are asked to choose and rank their top $k$ favorite items from the set of $n$ items. The efficient harmonic analysis of such \emph{partially ranked data} is addressed later in this paper (see Section~\ref{Section:  The Symmetric Group Acting on Tabloids}). Along the way, we also develop useful results concerning the harmonic analysis of finite-dimensional permutation modules in general, and we recover a well-known result due to Clausen \cite{clausen-1989} concerning a bound for the number of arithmetic operations necessary to apply generalized discrete Fourier transforms of finite groups (see Section~\ref{Section:  General Upper Bounds}).

The overarching approach that we take in this paper is relatively simple. Given a chain $\{1\} = G_1 < \cdots < G_n = G$ of subgroups of $G$, we will associate to each subgroup $G_j$ a special kind of harmonic basis $\mathcal{B}_j$ of $\mathbb{C}X$, where $\mathcal{B}_1$ is the standard basis of $\mathbb{C}X$. In particular, these harmonic bases will be \emph{symmetry adapted} (see Section~\ref{Section:  Symmetry Adapted Bases}). We will then bound the number of arithmetic operations necessary to do a change of basis from $\mathcal{B}_1$ to $\mathcal{B}_2$, then from $\mathcal{B}_2$ to $\mathcal{B}_3$, and so on, until we reach the change of basis from $\mathcal{B}_{n-1}$ to $\mathcal{B}_n$ (see Section~\ref{Section:  Bounds Based on the Dimensions of Frequency Spaces}).  Combining these results leads to an overall bound for the number of arithmetic operations necessary to do a change of basis from $\mathcal{B}_1$ to $\mathcal{B}_n$. As we will show, in some cases, this bound is surprisingly small.  

Using adapted bases to create fast algorithms for doing harmonic analysis is not new (see, for example, \cite{burgisser-etal-1997, clausen-and-baum-1993, maslen-rockmore-wolff-2018, maslen-rockmore-1997}). However, most books and papers on the subject focus primarily on the regular representation of a group (i.e., when $X = G$, and the action of $G$ on $X$ is simply given by group multiplication). In this paper, we show that it can also be fruitful to use adapted bases when dealing with other permutation representations. In particular, by using adapted bases that respect the intrinsic orbital structure of a permutation representation, we are able to provide straightforward but nontrivial bounds expressed in terms of the multiplicities of the irreducible representations that arise when we restrict the action of $G$ to subgroups in the chain $\{1\} = G_1 < \cdots < G_n = G$ (see Section~\ref{Section:  General Upper Bounds}). 

Although we hope that our results appeal to a wide-ranging audience, for both convenience and space considerations, we will assume throughout the rest of the paper that the reader has a working knowledge of the basic representation theory of the symmetric group, and that the reader is familiar with the representation theory of finite groups in general.  See, for example, the books by Sagan \cite{sagan-2001} and Serre \cite{serre-1977}.


\section{Background and Lemmas}\label{Section:  Background and Lemmas}  

In this section, we explain some of the terminology and notation we will be using throughout the rest of the paper. We also introduce a couple of foundational lemmas that will be used in the next section when we discuss general upper bounds.  Good references for the ideas addressed in this section are \cite{ceccherini-silberstein-etal-2008, clausen-and-baum-1993, fassler-and-stiefel-1992, terras-1999}.

\subsection{Functions on $X$}

Let $X = \{ x_1, \dots, x_m \}$ be a finite set.  We denote the complex vector space of complex-valued functions defined on $X$ by $\mathbb{C}X$.  For convenience, we will identify $x \in X$ with the function that is 1 on $x$ and $0$ on all of the other elements.  The $x_i$ then form a basis for $\mathbb{C}X$, which we call the \emph{standard basis} for $\mathbb{C}X$ and denote by $\mathcal{B}_1$.  

If $f = f(x_1) x_1 + \cdots + f(x_m) x_m$, then the coordinate vector of $f$ with respect to $\mathcal{B}_1$ is the column vector 
\[
[f]_{\mathcal{B}_1} = 
\begin{bmatrix}
f(x_1) \\ \vdots \\ f(x_m)
\end{bmatrix}. 
\]
More generally, if $\mathcal{B} = \{ b_1, \dots, b_m \}$ is a basis for $\mathbb{C}X$, and $f = \beta_1 b_1 + \cdots + \beta_m b_m$ for some complex coefficients $\beta_1, \dots, \beta_m$, then the coordinate vector of $f$ with respect to $\mathcal{B}$ is 
\[
[f]_\mathcal{B} = 
\begin{bmatrix}
\beta_1 \\ \vdots \\ \beta_m
\end{bmatrix}. 
\]

\subsection{Permutation Modules}

If $G$ is a finite group acting on the left on a finite set $X = \{ x_1, \dots, x_m \}$, then $G$ acts naturally on $\mathbb{C}X$, where if $g \in G$ and $f \in \mathbb{C}X$, then 
\[
(g \cdot f) (x) = f ( g^{-1} x)
\]
for all $x \in X$.  Linearly extending this action to $\mathbb{C}G$ then turns $\mathbb{C}X$ into a $\mathbb{C}G$-\emph{permutation module}.  The action of $G$ on $\mathbb{C}X$ gives rise to a \emph{permutation representation} 
\[
\varphi : G \to GL_{|X|}(\mathbb{C})
\]
where $\varphi(g)$ is the permutation matrix that encodes the action of $g \in G$ on $X$ with respect to the standard basis $\mathcal{B}_1$:  
\[
[\varphi(g)]_{ij} =
\begin{cases}
1 & \textnormal{if $g x_j = x_i$ } \\
0 & \textnormal{otherwise}.  
\end{cases}
\]
In particular, for all $g \in G$ and for all $f \in \mathbb{C}X$, 
\[
[g \cdot f]_{\mathcal{B}_1} = \varphi(g) [f]_{\mathcal{B}_1}. 
\]

In general, if $\mathcal{B}$ is a basis for $\mathbb{C}X$, then we will denote the matrix encoding of the action of $g \in G$ with respect to $\mathcal{B}$ by $[g]_\mathcal{B}$, which is the unique matrix with the property that for all $f \in \mathbb{C}X$, 
\[
[g]_\mathcal{B} [f]_\mathcal{B} = [g \cdot f]_\mathcal{B}.
\]
In this case, note that $\varphi(g) = [g]_{\mathcal{B}_1}$, and that the map $g \mapsto [g]_\mathcal{B}$ creates a representation of $G$ that is equivalent to $\varphi$ but uses the basis $\mathcal{B}$ instead of $\mathcal{B}_1$.

\subsection{Symmetry Adapted Bases}\label{Section:  Symmetry Adapted Bases}

Let $G$ be a finite group, and let $\mathcal{R}(G)$ denote a fixed maximal set of pairwise inequivalent irreducible complex representations of $G$. Let $M$ be a finite-dimensional $\mathbb{C}G$-module with basis $\mathcal{B}$. Recall that $\mathcal{B}$ is a \emph{harmonic basis} if it can be partitioned into subsets such that each subset forms a basis for an irreducible $\mathbb{C}G$-submodule of $M$.  We say $\mathcal{B}$ is a \emph{symmetry adapted basis of $M$ with respect to $\mathcal{R}(G)$} if for all $g \in G$, $[g]_\mathcal{B}$ is block diagonal, and each block of $[g]_\mathcal{B}$ is of the form $\rho(g)$ for some $\rho \in \mathcal{R}(G)$. A symmetry adapted basis is therefore a special kind of harmonic basis that reflects the action of $G$ as encoded in the representations in $\mathcal{R}(G)$.

Suppose now that $\{1\} = G_1 < \cdots < G_n = G$ is chain of subgroups of $G$. It is possible to find $\mathcal{R}(G_1), \dots, \mathcal{R}(G_n)$ and a basis $\mathcal{B}$ of $M$ such that for all $1 \le i \le n$, when $M$ is viewed as a $\mathbb{C}G_i$-module (by restricting the action of $G$ to $G_i$), $\mathcal{B}$ is simultaneously an adapted basis with respect to $\mathcal{R}(G_1), \dots, \mathcal{R}(G_n)$. In this case, we say that the $\mathcal{R}(G_i)$ are \emph{compatible} with respect to $M$, and that $\mathcal{B}$ is a \emph{symmetry adapted basis with respect to the list $\mathcal{R}(G_1), \dots, \mathcal{R}(G_n)$}.

In what follows, we will make use of harmonic bases $\mathcal{B}_1, \dots, \mathcal{B}_n$ where $\mathcal{B}_j$ is a symmetry adapted basis with respect to the list $\mathcal{R}(G_1), \dots, \mathcal{R}(G_j)$. The goal will be to show that the change of basis from $\mathcal{B}_1$ to $\mathcal{B}_n$ can sometimes be computed efficiently by computing a change of basis from $\mathcal{B}_1$ to $\mathcal{B}_2$, then from $\mathcal{B}_2$ to $\mathcal{B}_3$, then from $\mathcal{B}_3$ to $\mathcal{B}_4$, and so on until we reach $\mathcal{B}_n$. 

Since we are dealing with permutation representations in this paper, we will take advantage of the orbit structure that is inherently present. To explain, suppose the action of $G$ on $X$ partitions $X$ into orbits $X_1, \dots, X_t$. With a slight abuse of notation, we may view each $\mathbb{C}X_i$ as a submodule of $\mathbb{C}X$, in which case we may write $\mathbb{C}X$ as the direct sum
\[
\mathbb{C}X = \mathbb{C}X_1 \oplus \dots \oplus \mathbb{C}X_t.  
\]
An \emph{orbital harmonic basis} for $\mathbb{C}X$ is then a basis for $\mathbb{C}X$ that can be partitioned into subsets that form harmonic bases for the $\mathbb{C}X_i$. In what follows, we will assume that we are always dealing with orbital harmonic bases when we work with harmonic bases for permutation modules. As we will see, this will sometimes lead to a fast algorithm for doing harmonic analysis on a permutation module.

As a simple but helpful example, consider the situation where $G = S_3$ acts on the set $X = \{ 1, 2, 3 \}$ in the usual way. Expressing the vectors in $\mathbb{C}X$ using the coordinate vectors of the standard basis, we can write $\mathbb{C}X$ as a direct sum of irreducible $\mathbb{C}S_1$-, $\mathbb{C}S_2$-, and $\mathbb{C}S_3$-modules, respectively, as
\begin{align*}
\mathbb{C}X &= 
\langle \begin{bmatrix}1 \\ 0 \\ 0 \end{bmatrix} \rangle \oplus \langle \begin{bmatrix}0 \\ 1 \\ 0 \end{bmatrix} \rangle \oplus \langle \begin{bmatrix}0 \\ 0\\ 1 \end{bmatrix} \rangle
\\ 
&= \langle \begin{bmatrix}1 \\ 1 \\ 0 \end{bmatrix} \rangle \oplus \langle \begin{bmatrix}1 \\ -1 \\ 0 \end{bmatrix} \rangle \oplus \langle \begin{bmatrix}0 \\ 0 \\ 1 \end{bmatrix} \rangle
\\
&= \langle \begin{bmatrix}1 \\ 1 \\ 1 \end{bmatrix} \rangle \oplus \langle \begin{bmatrix}1 \\ -1 \\ 0 \end{bmatrix} , \begin{bmatrix}1/2 \\ 1/2 \\ -1 \end{bmatrix} \rangle. 
\end{align*}
The vectors appearing in each line above form orbital harmonic bases $\mathcal{B}_1, \mathcal{B}_2$, and $\mathcal{B}_3$ of $\mathbb{C}X$ with respect to the compatible \emph{seminormal representations} (see, for example, \cite{james-and-kerber-1981}) of $S_1$, $S_2$, and $S_3$, respectively. The change of basis matrix from $\mathcal{B}_1$ to $\mathcal{B}_2$ is 
\[
\begin{bmatrix}
1/2 &  1/2 &  0 \\
1/2 & -1/2 & 0 \\
0 & 0 & 1 
\end{bmatrix}
\]
and the change of basis matrix from $\mathcal{B}_2$ to $\mathcal{B}_3$ is 
\[
\begin{bmatrix}
2/3 &  0 &  1/3 \\
0 & 1 & 0 \\
2/3 & 0 & -2/3 
\end{bmatrix}.
\]
The change of basis matrix from $\mathcal{B}_1$ to $\mathcal{B}_3$ is therefore the product
\[
\begin{bmatrix}
2/3 &  0 &  1/3 \\
0 & 1 & 0 \\
2/3 & 0 & -2/3 
\end{bmatrix}
\begin{bmatrix}
1/2 &  1/2 &  0 \\
1/2 & -1/2 & 0 \\
0 & 0 & 1 
\end{bmatrix}
=
\begin{bmatrix}
1/3 &  1/3 &  1/3 \\
1/2 & -1/2 &  0 \\
1/3 &  1/3 & -2/3 
\end{bmatrix}.
\]

\subsection{Frequency Spaces}

Let $\rho \in \mathcal{R}(G)$, and suppose $\rho$ has degree $d$. For each $i$ such that $1 \le i \le d$, there is a primitive idempotent $e = \sum_{g \in G} \epsilon(g) g \in \mathbb{C}G$ associated to $\rho$ such that 
\begin{enumerate}
    \item $\sum_{g \in G} \epsilon(g) \rho(g)$ is a $d \times d$ matrix filled with zeros except for a 1 in the $ii$-th position, and 
    \item $\sum_{g \in G} \epsilon(g) \rho'(g)$ is the zero matrix for all $\rho' \in \mathcal{R}(G)$ such that $\rho' \ne \rho$.
\end{enumerate}
It follows that if $\mathcal{B}$ is a symmetry adapted basis with respect to $\mathcal{R}(G)$ for the $\mathbb{C}G$-module $M$, then each basis vector in $\mathcal{B}$ is an eigenvector for $e$ with eigenvalue 1 or 0. 

We call the subspace of $M$ spanned by the vectors in $\mathcal{B}$ that are eigenvectors for $e$ with eigenvalue 1 the \emph{frequency space} corresponding to $e$.  It is the subspace
\[
eM = \{ e m \ | \ m \in M \}.
\]
If we let $\mathcal{E}(G)$ denote the set of primitive idempotents in $\mathbb{C}(G)$ corresponding to $\mathcal{R}(G)$, then we may write $M$ as a direct sum
\[
M = \bigoplus_{e \in \mathcal{E}(G)} eM
\]
of frequency spaces. (This is known as a \emph{Pierce decomposition} of the module $M$.  See, for example,  \cite{drozd-kirichenko-1994}.) In this case, we will say that $eM$ is a \emph{frequency space of $M$ with respect to $\mathcal{R}(G)$}.

\begin{lemma}\label{freqspacelemma}
Let $G$ be a finite group, and let $M$ be a finite-dimensional $\mathbb{C}G$-module.  Suppose $\mathcal{R}(G) = \{\rho_1, \dots, \rho_h\}$ and that $U_j$ is an irreducible $\mathbb{C}G$-module corresponding to $\rho_j$. If $e$ is a primitive idempotent associated to $\rho_j$ as described above, and $M \cong \alpha_1 U_1 \oplus \cdots \oplus \alpha_h U_h$, then the frequency space $eM$ has dimension $\alpha_j$.
\end{lemma}

\begin{proof}
Let $\mathcal{B}$ be a symmetry adapted basis of $M$ with respect to $\mathcal{R}(G)$.  If $e = \sum_{g \in G} \epsilon(g) g$, then by construction, the matrix $[e]_\mathcal{B} = \sum_{g \in G} \epsilon(g) [g]_\mathcal{B}$ will have zeros everywhere except for $\alpha_j$ $1$'s on its diagonal, one for each copy of $U_j$ in $M$. It follows that $[e]_\mathcal{B}$ has rank $\alpha_j$, and thus the dimension of $eM$ is $\alpha_j$. 
\end{proof}

Suppose now that $H$ is a subgroup of $G$, and that $\mathcal{B}$ is a symmetry adapted basis with respect to $\mathcal{R}(H)$ and $\mathcal{R}(G)$. The following lemma shows that the frequency spaces for $H$ and $G$ are nicely related.  

\begin{lemma}\label{spacedecomp}
Let $G$ be a finite group, let $H$ be a subgroup of $G$, and let $M$ be a finite-dimensional $\mathbb{C}G$-module.  If $\mathcal{R}(H)$ and $\mathcal{R}(G)$ are compatible with respect to $M$, then each frequency space of $M$ with respect to $\mathcal{R}(H)$ is a direct sum of frequency spaces of $M$ with respect to $\mathcal{R}(G)$.
\end{lemma}

\begin{proof}
Without loss of generality, we may assume that $\mathcal{R}(H)$ and $\mathcal{R}(G)$ are also compatible with respect to the regular $\mathbb{C}G$-module. The result then follows from the fact that, by our definition of compatibility, each primitive idempotent in $\mathcal{E}(H)$ is a sum of primitive idempotents in $\mathcal{E}(G)$. 
\end{proof}

Finally, we introduce some notation and one more term that will be helpful in the next section.  

Suppose $G$ acts transitively on a finite set $X$. Let $H$ be a subgroup of $G$, let $N_1, \dots, N_s$ be a complete collection of irreducible $\mathbb{C}H$-modules, and suppose that as a $\mathbb{C}H$-module, we have that 
\[
\mathbb{C}X \cong \kappa_1 N_1 \oplus \cdots \oplus \kappa_s N_s.
\]
We will let $K_X(G,H)$ denote $\max\{ \kappa_1, \dots, \kappa_s \}$. Also, when we consider the frequency spaces of $\mathbb{C}X$ as a $\mathbb{C}H$-module, we will let $\Phi_X(G,H)$ denote the sum of the squares of the dimensions of these frequency spaces. By Lemma~\ref{freqspacelemma}, it follows that
\[
\Phi_X(G,H) = \kappa_1^2 \dim(N_1) + \cdots + \kappa_s^2 \dim(N_s).
\]

If the action of $G$ on $X$ is not transitive, and the orbits are $X_1, \dots, X_t$, then we will use $\Phi_X(G,H)$ to denote the sum 
\[
\Phi_{X_1}(G,H) + \cdots + \Phi_{X_t}(G,H)
\]
where we apply the previous definition of $\Phi_{X_i}(G,H)$ to each of the $X_i$. Similarly, we will use $K_X(G,H)$ to instead denote $\max \{ K_{X_1}(G,H), \dots, K_{X_t}(G,H) \}$. Finally, we will refer to the frequency spaces of the $\mathbb{C}X_i$ with respect to $\mathcal{R}(H)$ as \emph{$(G,H)$-orbital frequency spaces}.


\section{General Upper Bounds}\label{Section:  General Upper Bounds}

Let $G$ be a finite group acting on a finite set $X$, let $\mathbb{C}X$ be the associated $\mathbb{C}G$-permutation module, and let $f \in \mathbb{C}X$.  In this section, we will assume we are given the subgroup chain
\[
\{1\} = G_1 < \cdots < G_n = G
\]
and collections $\mathcal{R}(G_1), \dots, \mathcal{R}(G_n)$ of irreducible representations that are compatible with respect to $\mathbb{C}X$.  We will also assume that we are given harmonic bases $\mathcal{B}_1, \dots, \mathcal{B}_n$ of $\mathbb{C}X$, where 
\begin{enumerate}
    \item $\mathcal{B}_j$ is an orbital symmetry adapted basis with respect to $\mathcal{R}(G_1), \dots, \mathcal{R}(G_j)$, and 
    \item $\mathcal{B}_1$ is the standard basis of $\mathbb{C}X$.
\end{enumerate}

In this section, we find bounds for the number of arithmetic operations necessary to compute the coordinate vector $[f]_{\mathcal{B}_n}$ when given the coordinate vector $[f]_{\mathcal{B}_1}$.

\subsection{Nonzero Entries in Matrices}

If $A$ is a nonzero $m \times m$ matrix with complex entries, then we will denote the number of nonzero entries in $A$ by $\nu(A)$.  We will view $\nu(A)$ as a measure of how difficult it is to multiply the matrix $A$ and an arbitrary vector in $\mathbb{C}^m$.  In particular, note that such a product will require no more than $\nu(A)$ multiplications and strictly fewer than $\nu(A)$ additions.

If we let $\omega(A)$ denote the total number of arithmetic operations required to compute the product of a nonzero $m \times m$ matrix $A$ and an arbitrary vector in $\mathbb{C}^m$, note that we then have that $\omega(A) < 2 \nu(A)$.  Furthermore, if $A$ can be factored as a product $A = A_1 \cdots A_l$, then $\omega(A) < 2 \nu(A_1) + \cdots + 2 \nu(A_l)$.

Suppose now that $V$ is an $m$-dimensional vector space, and that $\mathcal{B}$ and $\mathcal{B'}$ are bases for $V$.  Let $C(\mathcal{B}, \mathcal{B}')$ be the change of basis matrix from the basis $\mathcal{B}'$ to the basis $\mathcal{B}$.  In other words, $C(\mathcal{B}, \mathcal{B}')$ is the unique $m \times m$ matrix such that
\[
C(\mathcal{B}, \mathcal{B}') [v]_{\mathcal{B}'} = [v]_\mathcal{B}
\]
for all $v \in V$.  In this case, note that $\nu( C(\mathcal{B}, \mathcal{B}') ) \le m^2$.  Also, if $\mathcal{B}''$ is another basis for $V$, then $C(\mathcal{B}, \mathcal{B}'') = C(\mathcal{B}, \mathcal{B}') C(\mathcal{B}', \mathcal{B}'')$ and thus
\[
\omega( C(\mathcal{B}, \mathcal{B}'') ) <  2 \nu( C(\mathcal{B}, \mathcal{B}') ) + 2 \nu( C(\mathcal{B}', \mathcal{B}'') ).
\]

It follows that if we are given harmonic bases $\mathcal{B}_1, \dots, \mathcal{B}_n$ of $\mathbb{C}X$ corresponding to the chain
\[
\{1\} = G_1 < \cdots < G_n = G
\]
then we may compute $[f]_{\mathcal{B}_n}$ by iteratively computing the product $C(\mathcal{B}_{j}, \mathcal{B}_{j-1}) [f]_{\mathcal{B}_{j-1}} = [f]_{\mathcal{B}_j}$ as $j$ goes from $2$ to $n$. Furthermore, this approach will be more efficient than the usual naive approach of simply computing $C(\mathcal{B}_n,\mathcal{B}_1)[f]_{\mathcal{B}_1} = [f]_{\mathcal{B}_n}$ whenever
\[
\nu(C(\mathcal{B}_n, \mathcal{B}_{n-1})) + \cdots + \nu(C(\mathcal{B}_{2}, \mathcal{B}_{1}))
\]
is small relative to $(\dim \mathbb{C}X)^2 = | X |^2$. We would therefore like to find upper bounds for each of the $\nu(C(\mathcal{B}_{j}, \mathcal{B}_{j-1}))$ that are much smaller than $|X|^2$.

\subsection{Bounds Based on the Dimensions of Frequency Spaces}\label{Section:  Bounds Based on the Dimensions of Frequency Spaces}

The following theorem is our main theorem. It provides a bound on $\nu(C(\mathcal{B}_j, \mathcal{B}_{j-1}))$ in terms of the dimensions of the frequency spaces of $\mathbb{C}X$ when viewed as a $\mathbb{C}G_{j-1}$-module. 

\begin{theorem}\label{cobbound}
Let $2 \le j \le n$. The number of nonzero entries in the change of basis matrix $C(\mathcal{B}_j, \mathcal{B}_{j-1})$ is bounded above by $\Phi_X(G_j,G_{j-1})$. In other words, 
\[
\nu(C(\mathcal{B}_j, \mathcal{B}_{j-1})) \le \Phi_X(G_j,G_{j-1}).
\]
Furthermore, each column of $C(\mathcal{B}_j, \mathcal{B}_{j-1})$ has at most $K_X(G_j,G_{j-1})$ nonzero entries, and thus we also have that 
\[
\nu( C(\mathcal{B}_j, \mathcal{B}_{j-1}) ) \le K_X(G_j, G_{j-1}) |X|.
\]
\end{theorem}

\begin{proof}
This follows from Lemma~\ref{freqspacelemma}, Lemma~\ref{spacedecomp}, and the fact that $\mathcal{B}_{j-1}$ and $\mathcal{B}_j$ are orbital symmetry adapted bases. In particular, suppose the action of $G_j$ on $X$ partitions $X$ into orbits $X_1, \dots, X_t$.  By Lemma~\ref{spacedecomp}, the change of basis from $\mathcal{B}_{j-1}$ to $\mathcal{B}_j$  only involves doing a change of basis within every frequency space, when viewed as a $\mathbb{C}G_{j-1}$-module, of each submodule $\mathbb{C}X_i$ of $\mathbb{C}X$. Thus, by Lemma~\ref{freqspacelemma}, each column of $C(\mathcal{B}_j, \mathcal{B}_{j-1})$ will have at most $K_X(G_j,G_{j-1})$ nonzero entries. Also, since a change of basis matrix for a $d$-dimensional vector space has at most $d^2$ nonzero entries, the number of nonzero entries in the change of basis matrix $C(\mathcal{B}_j, \mathcal{B}_{j-1})$ is bounded above by $\Phi_X(G_j,G_{j-1})$, which is simply the sum of the squares of the dimensions of the $(G_j, G_{j-1})$-orbital frequency spaces for $\mathbb{C}G_{j-1}$ (in each $\mathbb{C}X_i$).
\end{proof}

\begin{corollary}\label{together}
If $f\in \mathbb{C}X$, then the number $\omega(C(\mathcal{B}_n,\mathcal{B}_1))$ of arithmetic operations necessary to compute $[f]_{\mathcal{B}_n}$ when given $[f]_{\mathcal{B}_1}$ is strictly bounded above by
\[
2\Phi_X(G_n,G_{n-1}) + \cdots + 2\Phi_X(G_2,G_1).
\]
\end{corollary}
\begin{proof}
Recall that $C(\mathcal{B}_n,\mathcal{B}_1)[f]_{\mathcal{B}_1}=[f]_{\mathcal{B}_n}$. The result then follows directly from the fact that 
\[
\omega(C(\mathcal{B}_n,\mathcal{B}_1))\leq \omega(C(\mathcal{B}_{n},\mathcal{B}_{n-1})) + \cdots + \omega(C(\mathcal{B}_2,\mathcal{B}_1))
\]
and because by Theorem \ref{cobbound} we have that
\[
\omega(C(\mathcal{B}_{j},\mathcal{B}_{j-1})) < 2\nu(C(\mathcal{B}_{j},\mathcal{B}_{j-1})) \leq 2\Phi_X(G_j,G_{j-1})
\]
for all $j$ such that $2 \le j \le n$.
\end{proof}

\subsection{Regular Representations}

Before we apply the results above to the harmonic analysis of partially ranked data in the next section, we finish this section by considering the case where the permutation module in question is the regular representation of $G$. 

Suppose $H$ is a finite group, and that $d_1,\dots, d_s$ are the dimensions of the irreducible representations of $H$. Define $d^3(H)=d_1^3+\dots+ d_s^3$, and let $[G:H]$ denote the index of $H$ in $G$. 

\begin{lemma}\label{onestep}
Let $G$ be a finite group acting on $X = G$ by left multiplication so that $\mathbb{C}X=\mathbb{C}G$ is the regular $\mathbb{C}G$-module. If $H$ is a subgroup of $G$, then $\Phi_G(G,H)=[G:H]^2d^3(H)$.
\end{lemma}
\begin{proof}
Let $N_1,\dots, N_s$ be a complete collection of irreducible $\mathbb{C}H$-modules. The orbits of $G$ under the action of $H$ are the right cosets of $H$ in $G$. Therefore, as a $\mathbb{C}H$-module, $\mathbb{C}G\cong \bigoplus_{i=1}^{[G:H]}\mathbb{C}H$. As a $\mathbb{C}H$-module, $\mathbb{C}H \cong \bigoplus_{i = 1}^s \dim(N_i)N_i$, so we therefore have that 
\[
\mathbb{C}G \cong [G:H]\dim(N_1) N_1 \oplus \cdots \oplus [G:H]\dim(N_s) N_s.
\]
The result now follows from Lemma~\ref{freqspacelemma} and the discussion after Lemma~\ref{spacedecomp}.
\end{proof}

By recursively applying Lemma~\ref{onestep} to a chain of subgroups, we get the following theorem, which is essentially Theorem~1.1 of \cite{clausen-1989}, and which provides a general upper bound for applying a discrete Fourier transform to a finite group and therefore a bound for doing harmonic analysis on the regular representation of a finite group.  

\begin{theorem}
Let $G$ be a finite group acting on $X = G$ by left multiplication so that $\mathbb{C}X=\mathbb{C}G$ is the regular $\mathbb{C}G$-module. Let 
\[
\{1\} = G_1 < \cdots < G_n = G
\]
be a chain of subgroups of $G$, and suppose $\mathcal{R}(G_1), \dots, \mathcal{R}(G_n)$ are compatible with respect to $\mathbb{C}X$. Let $\mathcal{B}_1, \dots, \mathcal{B}_n$ be harmonic bases of $\mathbb{C}X$, where $\mathcal{B}_j$ is an orbital symmetry adapted basis with respect to $\mathcal{R}(G_1), \dots, \mathcal{R}(G_j)$, and $\mathcal{B}_1$ is the standard basis of $\mathbb{C}X$. If $q_j = [G_j:G_{j-1}]$, then 
\[
\omega(C(\mathcal{B}_n,\mathcal{B}_1)) < 2\sum_{j=2}^n(q_j^2q_{j+1}\cdots q_nd^3(G_{j-1})).
\]
\end{theorem}
\begin{proof}
By Theorem~\ref{cobbound} and Lemma~\ref{onestep}, the number of nonzero entries in $C(\mathcal{B}_j,\mathcal{B}_{j-1})$ is bounded above by
\[
\Phi_G(G_j,G_{j-1}) = 
[G:G_j][G_j:G_{j-1}]^2d^3(G_{j-1}) = q_n\cdots q_{j+1}q_j^2d^3(G_{j-1}).
\]
The result now follows from Corollary~\ref{together}.  
\end{proof}


\section{The Symmetric Group Acting on Tabloids}\label{Section:  The Symmetric Group Acting on Tabloids}

We now shift our attention to the permutation modules that arise when the symmetric group acts on \emph{tabloids} (see below for the definition). These objects have been used to index both fully and partially ranked data (see, for example, \cite{crisman-orrison-2017, daugherty-etal-2009}). The results in this section may therefore be viewed as statements about the efficient analysis of such data. Good references for many of the ideas and much of the notation found in this section are \cite{james-and-kerber-1981, sagan-2001, stanley-1999}.

\subsection{Tabloids}

Let $n$ be a positive integer. A \emph{weak composition} of $n$ is a sequence $\alpha=(\alpha_1,\dots,\alpha_k)$ of nonnegative integers satisfying $\sum \alpha_i=n$. If each of the summands $\alpha_1,\dots,\alpha_k$ is a positive integer, then we say $\alpha$ is a \emph{composition} of $n$. A \emph{partition} of $n$ is then a composition $\lambda=(\lambda_1,\dots,\lambda_k)$ of $n$ such that $\lambda_1 \geq \cdots \geq \lambda_k$. If $\lambda=(\lambda_1,\dots,\lambda_k)$ is a partition of $n$ then we write $\lambda\vdash n$ and say that $\lambda$ has $k$ \emph{parts}. 

Note that a weak composition $\alpha$ can be identified with a unique partition $\bar{\alpha}$ by writing the positive summands in weakly decreasing order. For example, if $\alpha=(3,4,0,1,0,2)$, then
\[
\bar{\alpha}=(4,3,2,1).
\]
We will make use of this fact below. 

Given a weak composition $\alpha = (\alpha_1,\dots,\alpha_k)$ of $n$, the \emph{Young diagram of shape} $\alpha$ is the left-justified array of boxes with $k$ rows and $\alpha_i$ boxes in the $i$th row. Filling these boxes with the numbers $1,\dots,n$, without repetition, creates a \emph{Young tableau of shape} $\alpha$. Two Young tableaux of shape $\alpha$ are then said to be \emph{row-equivalent} if they have the same set of numbers in each row. Each equivalence class of tableaux of shape $\alpha$ under this relation is a \emph{tabloid of shape} $\alpha$. We will use $X^\alpha$ to denote the set of all tabloids of shape $\alpha$. 

It is common to denote a tabloid by first forming a representative tableau and then removing the vertical dividers within each row (see Figure \ref{tabloid}).

\begin{figure}[h]
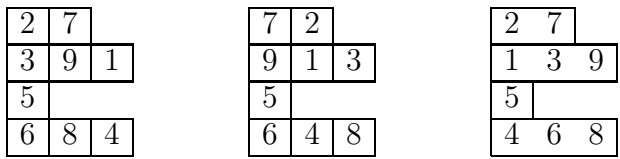

\begin{center}
\begin{tabular}{| c | c | c |} \cline{1-2}
2 & 7 & \multicolumn{1}{| c}{ } \\ \cline{1-3} 
3 & 9 & 1 \\ \cline{1-3}
5 & \multicolumn{2}{| c}{} \\ \cline{1-3}
6 & 8 & 4 \\ \cline{1-3}
\end{tabular}
\hspace{.5in}
\begin{tabular}{| c | c | c |} \cline{1-2}
7 & 2 & \multicolumn{1}{| c}{ } \\ \cline{1-3} 
9 & 1 & 3 \\ \cline{1-3}
5 & \multicolumn{2}{| c}{} \\ \cline{1-3}
6 & 4 & 8 \\ \cline{1-3}
\end{tabular}
\hspace{.5in} 
\begin{tabular}{| c c c |} \cline{1-2}
2 & 7 & \multicolumn{1}{| c}{ } \\ \cline{1-3} 
1 & 3 & 9 \\ \cline{1-3}
5 & \multicolumn{2}{| c}{} \\ \cline{1-3}
4 & 6 & 8 \\ \cline{1-3}
\end{tabular}
\end{center}
\caption{Two equivalent tableaux of shape $(2,3,1,3)$ and their tabloid.}
\label{tabloid}
\end{figure}

\subsection{The Action of $S_n$}
We may view $\mathbb{C}X^\alpha$ as a $\mathbb{C}S_n$-permutation module under the natural action of the symmetric group $S_n$ on $X^\alpha$, where if $\sigma\in S_n$ and $T\in X^\alpha$, then  $\sigma\cdot T$ is the tabloid obtained from $T$ by applying the permutation $\sigma$ to each entry of $T$ (see Figure \ref{tabact}). 

\begin{figure}[h]
\begin{center}
$(147)(56)\cdot$
\begin{tabular}{| c c c |} \cline{1-2}
2 & 7 & \multicolumn{1}{| c}{ } \\ \cline{1-3} 
1 & 3 & 9 \\ \cline{1-3}
5 & \multicolumn{2}{| c}{} \\ \cline{1-3}
4 & 6 & 8 \\ \cline{1-3}
\end{tabular}
=
\begin{tabular}{| c c c |} \cline{1-2}
2 & 1 & \multicolumn{1}{| c}{ } \\ \cline{1-3} 
4 & 3 & 9 \\ \cline{1-3}
6 & \multicolumn{2}{| c}{} \\ \cline{1-3}
7 & 5 & 8 \\ \cline{1-3}
\end{tabular}
\end{center}
\caption{The action of $\sigma=(147)(56)$ on a tabloid of shape $(2, 3, 1, 3)$.}
\label{tabact}
\end{figure}

Note that the action of $S_n$ on $X^\alpha$ does not depend on the order of the rows of $\alpha$. In particular, as $\mathbb{C}S_n$-modules, we have that $\mathbb{C}X^\alpha \cong \mathbb{C}X^{\bar{\alpha}}$. For convenience, we will therefore always assume that we are starting with a permutation module $\mathbb{C}X^\lambda$ where $\lambda$ is a partition.

Fortunately, there are well-studied collections $\mathcal{R}(S_1),\dots,\mathcal{R}(S_n)$ of irreducible representations that are pairwise compatible with respect to $\mathbb{C}X^\lambda$. For example, two possible such collections are Young's seminormal representations and Young's orthogonal representations. In fact, these are the representations most often used when doing harmonic analysis on tabloids, and these are the ones we suggest using if the reader wishes to implement any of the ideas that follow.

\subsection{Irreducible Representations}
There is a well-known parametrization of the irreducible representations of $S_n$ by the partitions of $n$. We use $S^\mu$ to denote the irreducible $\mathbb{C}S_n$-module corresponding to $\mu\vdash n$. ($S^\mu$ is called a \emph{Specht module}.) The multiplicity of $S^\mu$ in the decomposition of $\mathbb{C}X^\lambda$ into irreducible $\mathbb{C}S_n$-modules is given by the well-studied but elusive \emph{Kostka numbers}, which are defined below.

\begin{definition} Given two partitions $\lambda = (\lambda_1,\dots,\lambda_k),\mu=(\mu_1,\dots,\mu_{k'})$ of $n$, we say that $\mu$ \emph{dominates} $\lambda$, denoted $\lambda\unlhd\mu$, if
\[
\sum_{i=1}^j\lambda_i\leq \sum_{i=1}^j \mu_i
\]
for all $j$ such that $1\leq j\leq \max(k,k')$. If $\lambda\unlhd\mu$, then a \emph{Kostka filling of $\mu$ by $\lambda$} is any filling of a tableau of shape $\mu$ using exactly $\lambda_i$ $i$'s such that the entries in each row, when read from left to right, are non-decreasing and the entries of each column, when read from top to bottom, are strictly increasing.
\end{definition}
\begin{example} If $\lambda=(3,1,1)$ and $\mu=(4,1)$, then $\lambda\unlhd\mu$ and there are 2 Kostka fillings of $\mu$ by $\lambda$, namely
\[
\begin{array}{lll} 
\ytableausetup{centertableaux}
\begin{ytableau}
1 & 1 & 1& 2 \\
3
\end{ytableau} 
\textnormal{\hspace{.2in} and \hspace{.2in}}
\begin{ytableau}
1 & 1 & 1 & 3 \\
2
\end{ytableau}
\end{array}.
\]
\end{example}

Let $\kappa_{\mu,\lambda}$ denote the number of Kostka fillings of $\mu$ by $\lambda$. These are the \emph{Kostka numbers}, and Young's rule (see, for example, Theorem 2.11.2 in \cite{sagan-2001}) states that as a $\mathbb{C}S_n$-module, 
\begin{equation}\label{Young's Rule}\mathbb{C}X^\lambda\cong \bigoplus_{\mu: \, \lambda\unlhd\mu}\kappa_{\mu,\lambda}S^\mu.\end{equation}

In order to make use of Theorem \ref{cobbound} when dealing with $\mathbb{C}X^\lambda$, we need to find suitable bounds for $K_{X^{\lambda}}(S_j,S_{j-1})$ for each $2\leq j\leq n$. Recall that if the orbits of $X^{\lambda}$ under the action of $S_j$ are $X_1,\dots,X_t$, then $K_{X^{\lambda}}(S_j,S_{j-1})$ is the maximum over all of the $\mathbb{C}X_i$ of the dimensions of their $(S_j,S_{j-1})$-orbital frequency spaces. Combining an understanding of the orbits of $X_i$ under the action of $S_{j-1}$, and decomposing the resulting spaces into irreducible $\mathbb{C}S_{j-1}$-modules according to Young's rule will therefore enable us to determine the maximum frequency space dimension $K_{X^{\lambda}}(S_j,S_{j-1})$. We begin by considering the case where $j = n$, and then proceed recursively.

\subsection{Orbits Under the Action of $\mathbb{C}S_{n-1}$}

Suppose $\lambda \vdash n$. The orbits of $S_{n-1}$ acting on $X^\lambda$ are given by the ways to obtain a weak composition of $n-1$ from $\lambda$tracting 1 from a nonzero part $\lambda_i$ of $\lambda$. To describe this more formally, we introduce the following notation.
\begin{definition} If $\lambda=(\lambda_1,\dots,\lambda_k)$ is a partition, and $\alpha= (\lambda_1, \dots, \lambda_{i-1}, \lambda_i-1, \lambda_{i+1}, \dots, \lambda_k)$, then we will define $\lambda^i$ to be the partition $\overline{\alpha}$.
\end{definition}
\begin{example}
Let $\lambda=(4,2,1,1)$. Then $\lambda^1=(3,2,1,1)$, $\lambda^2=(4,1,1,1)$, and $\lambda^3=\lambda^4=(4,2,1)$.
\end{example}
\begin{theorem}\label{orbits} If $\lambda=(\lambda_1,\dots,\lambda_k)$ is a partition of $n$, then as a $\mathbb{C}S_{n-1}$-module, 
\[\mathbb{C}X^\lambda\cong\bigoplus_{i=1}^k\mathbb{C}X^{\lambda^i}.\]
\end{theorem}
\begin{proof}
Two tabloids $T,T'\in X^\lambda$ are in the same orbit under the action of $S_{n-1}$ if and only if they contain $n$ in the same row. Furthermore, the action of $S_{n-1}$ on the orbit of a tabloid $T$ that contains $n$ in its $i$th row corresponds exactly to the action of $S_{n-1}$ on $X^{\lambda^i}$. It follows that, as a $\mathbb{C}S_{n-1}$-module,  $\mathbb{C}X^\lambda\cong\bigoplus_{i=1}^k\mathbb{C}X^{\lambda^i}.$
\end{proof}

\subsection{The Case $\lambda=(n-k,k)$}
The action of the symmetric group on two-rowed tabloids gives rise to a permutation module with a particularly simple decomposition into irreducible modules according to using Young's rule. As noted in Section~\ref{Section:  Introduction}, these tabloids arise when dealing with survey data for which respondents have been asked simply to choose their top $k$ items from a set of $n$ items, where $k \le n/2$.

\begin{lemma}\label{simpleyoung} If $\lambda=(n-k,k)$ is a partition of $n$, then as a $\mathbb{C}S_n$-module, 
\[
\mathbb{C}X^\lambda\cong \bigoplus_{0\leq l\leq k} S^{(n-l,l)}.
\]
In particular, this is a multiplicity-free decomposition.
\end{lemma}
\begin{proof}
By Young's rule (\ref{Young's Rule}), we are only concerned with $S^\mu$ such that $\lambda\unlhd\mu$. Since $\lambda=(n-k,k)$, we see that $\mu$ must have the form $(n-l,l)$ for $0\leq l\leq k$. The multiplicity of $S^\mu$ is given by $\kappa_{\mu,\lambda}$, which is the number of Kostka fillings of $\mu=(n-l,l)$ with $(n-k)$ 1's and $k$ 2's. To be a Kostka filling, all $(n-k)$ 1's must be in the first row of $\mu$, completely determining the filling. Thus $\kappa_{\mu,\lambda}=1$.
\end{proof}

Together with Theorem~\ref{cobbound} and Theorem~\ref{orbits}, the simple decomposition of $\mathbb{C}X^{(n-k,k)}$ into irreducible modules in Lemma~\ref{simpleyoung} gives rise to an efficient approach to doing a change of basis from the standard basis $\mathcal{B}_1$ to a harmonic basis $\mathcal{B}_n$, especially when compared to the naive bound of $\omega(C(\mathcal{B}_n, \mathcal{B}_1)) < 2 \binom{n}{k}^2$. 

\begin{theorem}\label{Theorem: (n-k,k)} Suppose $\lambda=(n-k,k)$ is a partition of $n$, and that $\mathcal{R}(S_1),\dots,\mathcal{R}(S_n)$ are compatible with respect to $\mathbb{C}X^\lambda$. Let $\mathcal{B}_1,\dots,\mathcal{B}_n$ be harmonic bases of $\mathbb{C}X^\lambda$ where $\mathcal{B}_j$ is an orbital symmetry adapted basis with respect to $\mathcal{R}(S_1),\dots,\mathcal{R}(S_j)$, and $\mathcal{B}_1$ is the standard basis of $\mathbb{C}X^\lambda$. Then
\[
\omega(C(\mathcal{B}_n, \mathcal{B}_1)) < 4 (n-1) \binom{n}{k}. 
\]
\end{theorem}
\begin{proof}
Let $2 \le j \le n$. The action of $S_j$ on $X^\lambda$ partitions $X^\lambda$ into orbits $X^{\lambda'}$ that correspond to tabloids with at most two rows and with entries from $\{ 1, \dots, j \}$. With the use of Theorem~\ref{orbits} and Lemma~\ref{simpleyoung}, we see that restricting the action further to $S_{j-1}$ on $X^{\lambda'}$ gives rise to a decomposition of the $\mathbb{C}X^{\lambda'}$ into irreducible $\mathbb{C}S_{j-1}$-modules each with multiplicity no more than two. It follows that
\[
K_{X^\lambda}(S_j, S_{j-1}) \le 2. 
\]
By Theorem~\ref{cobbound}, we have that 
\[
\nu(C(\mathcal{B}_j,\mathcal{B}_{j-1})) \le K_{X^\lambda}(S_j, S_{j-1}) |X^\lambda| \le 2 \binom{n}{k}.
\]
The theorem statement now follows from the fact that 
\[
\omega(C(\mathcal{B}_n, \mathcal{B}_1)) < \sum_{j=2}^n 2 \nu (C(\mathcal{B}_j, \mathcal{B}_{j-1})). 
\]
\end{proof}

\subsection*{Remark}
Theorem~\ref{Theorem: (n-k,k)} is essentially Theorem~2 in \cite{iglesias-natale-2017}. In that paper, the authors use Gelfand-Tsetlin bases for what we are calling harmonic bases, and the bound they provide is $2(n-1)\binom{n}{k}$. It differs from our bound of $4(n-1)\binom{n}{k}$ because the computational model they use counts a single complex multiplication and addition as one operation.

\subsection{The Case $\lambda=(n-k,1,\dots,1)$}

We now consider the situation where we are dealing with partitions of $n$ of the form $\lambda = (n-k, 1, \dots, 1)$. As noted in Section~\ref{Section:  Introduction}, tabloids of this shape arise when dealing with survey data for which respondents have been asked to rank their top $k$ items from a set of $n$ items. 

As we saw when $\lambda=(n-k,k)$, the two main steps in determining bounds on the number of nonzero entries in a full factorization of the change of basis matrix $C(\mathcal{B}_n,\mathcal{B}_1)$ involve decomposing $X^\lambda$ into orbits $X^{\lambda'}$ under the action of $S_j$, and then decomposing the resulting permutation modules $\mathbb{C}X^{\lambda'}$ into irreducible $\mathbb{C}S_{j-1}$-submodules. Bounds then come from the maximum dimensions of the $(S_{j},S_{j-1})$-orbital frequency spaces, which amounts to summing Kostka numbers across certain $S_{j-1}$-orbits.

 Though the Kostka numbers are well-studied, Stanley remarks in \cite{stanley-1999} that it is unlikely that a general formula for $\kappa_{\mu,\lambda}$ exists. Fortunately, when $\lambda=(n-k,k)$, the Kostka number $\kappa_{\mu,\lambda}$ is trivially 1 for each $\mu$ dominating $\lambda$. When $\lambda=(n-k,1,\dots,1)$, this is no longer the case. However, the orbits in this case still exhibit enough structure for us to find bounds for the multiplicities. We use the language of \emph{skew tableaux} to describe the orbits (see, for example, \cite{james-and-kerber-1981}).
 
 Let $\lambda=(\lambda_1,\dots,\lambda_k)$ and $\mu=(\mu_1,\dots, \mu_\ell)$ be partitions. We say $\mu$ \emph{contains} $\lambda$ if $k\leq \ell$ and $\lambda_i\leq \mu_i$ for all $i$ such that $1\leq i\leq k$. If $\mu$ contains $\lambda$, then the \emph{skew diagram of shape} $\mu/\lambda$ is the set of boxes in the Young diagram of shape $\mu$ that are not in the Young diagram of shape $\lambda$. If $n$ is the number of boxes in the skew diagram of shape $\mu/\lambda$, then a \emph{standard skew tableau of shape} $\mu/\lambda$ is a filling of the boxes of the skew diagram with the numbers $1,\dots,n$, without repetition, so that the entries in each row and each column are strictly increasing (see Figure~\ref{fig: skewtab}). 

\begin{figure}
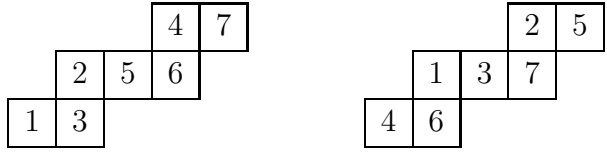

\begin{ytableau}
\none & \none & \none & 4 & 7 \\
\none & 2 & 5 & 6 \\
1 & 3 \\
\end{ytableau}
\hspace{.5in}
\begin{ytableau}
\none & \none & \none & 2 & 5 \\
\none & 1 & 3 & 7 \\
4 & 6 \\
\end{ytableau}
\caption{Two standard skew tableaux of shape $\mu/\lambda$ where $\mu  = (5,4,2)$ and $\lambda = (3,1)$.}
\label{fig: skewtab}
\end{figure}
  
Let $d_{\mu/\lambda}$ denote the number of standard skew tableau of shape $\mu/\lambda$. For convenience, if $\mu$ does not contain $\lambda$, then we will set $d_{\mu/\lambda} = 0$. For positive integers $k$ and $n$ such that $k \le n$, let $(n)_k$ denote the falling factorial
\[
(n)_k=n(n-1)(n-2)\cdots (n-(k-1)).
\]
Note that if $\lambda = (n-k,1,\dots,1)$ is a partition of $n$, then $\dim \mathbb{C}X^\lambda = (n)_k$.  

For nonnegative integers $\ell$ and $r$ such that $\ell\geq r$, let $D(\ell,r)=\{\mu\mid \mu\vdash \ell, \mu_1\geq r\}$.
Let \[M(n,k)=\max_{\mu\in D(n-1,n-k-1)}\left\{d_{\mu/(n-k-1)}+kd_{\mu/(n-k)}\right\}.\]
Note that if $\mu$ has first part $\mu_1=n-k-1$, the sum reduces to $d_{\mu/(n-k-1)}$.

As $\mu/(n-k-1)$  represents the skew diagram resulting from removing $n-k-1$ of the boxes in the first row of $\mu$, and similarly for $\mu/(n-k)$, for large enough $n$, the value of $M(n,k)$ depends only on $k$. For example, $M(7,6)=112$, $M(8,6)=155$, and $M(n,6)=160$ for all $n>8$. Let $N(k)$ denote the maximum value of $M(n,k)$ over all positive integers $n$ such that $n>k$.

\begin{theorem}\label{secondexthm} Let $\lambda=(n-k,1,\dots,1)\vdash n$ and let $\mathcal{R}(S_1),\dots,\mathcal{R}(S_n)$ be  compatible with respect to $\mathbb{C}X^\lambda$. Let $\mathcal{B}_1,\dots,\mathcal{B}_n$ be harmonic bases of $\mathbb{C}X^\lambda$ where $\mathcal{B}_j$ is an orbital symmetry adapted basis with respect to $\mathcal{R}(S_1),\dots,\mathcal{R}(S_j)$, and $\mathcal{B}_1$ is the standard basis of $\mathbb{C}X^\lambda$. Then
\[
\omega(C(\mathcal{B}_n, \mathcal{B}_1)) < 2N(k)(n-1)(n)_k 
\]
 as opposed to the naive bound of $2((n)_{k})^2$. 
\end{theorem}  
In Table \ref{table}, we compare the bound of Theorem \ref{secondexthm} to the naive bound.
\begin{proof}
Let $\lambda=(n-k,1,\dots,1)\vdash n$. By Theorem \ref{orbits}, as a $\mathbb{C}S_{n-1}$-module, 
\[\mathbb{C}X^\lambda\cong\bigoplus_{i=1}^{k+1} \mathbb{C}X^{\lambda^i}\cong \mathbb{C}X^{\lambda^1}\oplus k\mathbb{C}X^{\lambda^2}\]
where $\lambda^1=(n-k-1,1,\dots,1)\vdash (n-1)$ and $\lambda^2=\lambda^3=\dots=\lambda^{k+1}=(n-k,1,\dots,1)\vdash (n-1)$. Using Young's rule (\ref{Young's Rule}) to decompose $\mathbb{C}X^{\lambda^1}$ as a $\mathbb{C}S_{n-1}$-module we first note that the set of partitions of $n-1$ that dominate $\lambda^1$ is given by $D(n-1,n-k-1).$

Suppose $\mu  \in D(n-1,n-k-1)$. To determine the multiplicity of $S^\mu$ in the decomposition of $\mathbb{C}X^{\lambda^1}$ as a $\mathbb{C}S_{n-1}$-module, we need to determine $\kappa_{\mu,\lambda^1}$. This is the number of Kostka fillings of $\mu$ with  $(n-k-1)$ 1's, along with the entries $2,\dots, k+1$. In order to be a valid Kostka filling, this means that the first $n-k-1$ entries of the first row of $\mu$ must be a 1, leaving exactly one of each number $2,\dots,k+1$ to fill the remaining $k$ boxes of $\mu$. This is exactly the number of standard skew tableaux of shape $\mu/(n-k-1)$. Thus, as a $\mathbb{C}S_{n-1}$-module, 
\[
\mathbb{C}X^{\lambda^1}\cong\bigoplus_{\mu\in D(n-1,n-k-1)} d_{\mu/(n-k-1)} S^{\mu}.
\]

As $\lambda^2=(n-k,1,\dots,1)$, with $(k-1)$ 1's we see by the same argument that each of the $k$ copies of $\mathbb{C}X^{\lambda^2}$ decomposes as a $\mathbb{C}S_{n-1}$-module as
\[
\mathbb{C}X^{\lambda^2}\cong \bigoplus_{\mu\in D(n-1,n-k)} d_{\mu/(n-k)} S^{\mu}.
\]

Then in the complete decomposition of $\mathbb{C}X^\lambda$ into irreducible $\mathbb{C}S_{n-1}$-modules, only those $S^\mu$ with $\mu\in D(n-1,n-k-1)$ will appear, and they appear with multiplicity 
\[d_{\mu/(n-k-1)}+kd_{\mu/(n-k)}.\]
Note that if $\mu_1=(n-k-1)$, this sum reduces to $d_{\mu/(n-k-1)}$, as $d_{\mu/(n-k)}=0$. By Theorem~\ref{cobbound}, the number of nonzero entries in each column of the change of basis matrix $C(\mathcal{B}_{n},\mathcal{B}_{n-1})$ is bounded by $K_{X^\lambda}(S_n,S_{n-1})= M(n,k)$.

Now consider the change of basis matrix $\mathcal{C}(\mathcal{B}_{j},\mathcal{B}_{j-1})$. By repeated applications of Theorem \ref{orbits}, as a $\mathbb{C}S_j$-module,
\[
\mathbb{C}X^\lambda\cong\bigoplus\mathbb{C}X^{\lambda'}
\]
where the $\lambda'$ are a collection of (possibly repeated) partitions of $j$, each of which has form $\lambda'=(j-m,1,\dots,1)$ for some nonnegative integer $m$, where $k-(n-j)\leq m\leq k$.

By Theorem~\ref{cobbound}, we need only determine
\[
K_{X^\lambda}(S_j,S_{j-1})= \max_{\lambda'}K_{X^{\lambda'}}(S_j,S_{j-1}). 
\]
Applying the same arguments as above to decompose $\mathbb{C}X^{\lambda'}$ into irreducible $\mathbb{C}S_{j-1}$-submodules using Theorem~\ref{orbits} and Young's rule~(\ref{Young's Rule}), we see that the number of nonzero entries in each column of the change of basis matrix $C(\mathcal{B}_{j},\mathcal{B}_{j-1})$ is bounded by $M(j,m)$, where
\[
M(j,m)=\max_{\mu\in D(j-1,j-m-1)}\{d_{\mu/(j-m-1)}+md_{\mu/(j-m)}\}.
\]

Suppose $\mu=(\mu_1,\dots,\mu_s)\in D(j-1,j-m-1)$. Then $\mu\vdash j-1$ and $\mu_1=j-m-1+t$ for some $t$ such that $0\leq t\leq m$. The skew diagram $\mu/(j-m-1)$ has the same shape as $\mu$ for all but its first row and has $t$ boxes in its first row.

Let $\mu'=(\mu_1+n-j,\mu_2,\dots,\mu_s)$. Then $\mu'\vdash n-1$ and $\mu_1'=n-m-1+t\geq n-k-1+t$, so $\mu'\in D(n-1,n-k-1)$. The skew diagram $\mu'/(n-k-1)$ has the same shape as $\mu$ for all but its first row and has $t+k-m\geq t$ boxes its first row. Moreover, as $n-k-1\geq j-m-1$, the first box in this row is either in the same position or to the right of the first box in the first row of $\mu/(j-m-1)$.

A standard skew tableau of shape $\mu/(j-m-1)$ yields a standard skew tableau of shape $\mu'/(n-k-1)$ as follows: fill the boxes in rows $\mu_2,\dots, \mu_s$ the same way as in $\mu/(j-m-1)$ and fill the first $t$ boxes of the first row of $\mu'/(n-k-1)$ as they are filled in the first row of $\mu/(j-m-1)$. If boxes remain in the first row of $\mu'/(n-k-1)$, fill them with the numbers that remain, in ascending order from left to right.

This implies that there are at least as many standard skew tableau of shape  $\mu'/(n-k-1)$ as there are of shape $\mu/(j-m-1)$. Thus, $ d_{\mu/(j-m-1)}\leq d_{\mu'/(n-k-1)}$. Similarly, $ d_{\mu/(j-m)}\leq d_{\mu'/(n-k)}$ and so
\[
d_{\mu/(j-m-1)}+md_{\mu/(j-m)}\leq d_{\mu'/(n-k-1)}+kd_{\mu'/(n-k)}
\]
implying that $M(j,m)\leq M(n,k)\leq N(k)$. By Theorem~\ref{cobbound}, we have that 
\[
\nu(C(\mathcal{B}_j,\mathcal{B}_{j-1})) \le K_{X^\lambda}(S_j, S_{j-1}) |X^\lambda| \leq N(k) (n)_k.
\]
The theorem statement now follows from the fact that 
\[
\omega(C(\mathcal{B}_n, \mathcal{B}_1)) < \sum_{j=2}^n 2 \nu (C(\mathcal{B}_j, \mathcal{B}_{j-1})). 
\]
\end{proof}

\begin{table}
\begin{center}
\begin{tabular}{l|l|l}
$k$&$N(k)(n-1)(n)_k$ \\\hline
2&4$(n-1)(n)_2$&$((n)_2)^2$\\\hline 
3&9$(n-1)(n)_3$&$((n)_3)^2$\\\hline
4&18$(n-1)(n)_4$&$((n)_4)^2$\\\hline
5&60$(n-1)(n)_5$&$((n)_5)^2$\\\hline
6&160$(n-1)(n)_6$&$((n)_6)^2$\\\hline
7&420$(n-1)(n)_7$&$((n)_7)^2$\\\hline
8&1344$(n-1)(n)_8$&$((n)_8)^2$\\\hline
9&5376$(n-1)(n)_9$&$((n)_9)^2$\\\hline
10&16800$(n-1)(n)_{10}$&$((n)_{10})^2$\\\hline
11&59400$(n-1)(n)_{11}$&$((n)_{11})^2$\\\hline
12&222750$(n-1)(n)_{12}$&$((n)_{12})^2$\\\hline
13&878592$(n-1)(n)_{13}$&$((n)_{13})^2$\\\hline
\end{tabular}
\caption{Comparison of bound given by Theorem \ref{secondexthm} with naive bound}\label{table}
\end{center}
\end{table}


\section{Conclusion and Open Questions}
In this paper, we have developed a framework for computing the coefficients of $f\in\mathbb{C}X$ in terms of a harmonic basis $\mathcal{B}$ of $\mathbb{C}X$ using intermediate bases $\mathcal{B}_1,\dots,\mathcal{B}_n$ and iteratively computing a change of basis from $\mathcal{B}_{j-1}$ to $\mathcal{B}_{j}$. In Section \ref{Section:  General Upper Bounds}, we saw that when $\mathcal{B}_1$ is the standard basis and $\mathcal{B}_j$ is an orbital symmetry adapted basis with respect to compatible representations $\mathcal{R}(G_1),\dots, \mathcal{R}(G_n)$, we can bound the number of nonzero entries in the change of basis matrix $C(\mathcal{B}_j,\mathcal{B}_{j-1})$ in terms of the multiplicities of the irreducible $\mathbb{C}G_{j-1}$-submodules in the orbital decomposition of $\mathbb{C}X$ under the action of $G_j$. 

In Section \ref{Section:  The Symmetric Group Acting on Tabloids}, we applied these results to permutation modules that arise when the symmetric group acts on a set of tabloids $X^\lambda$. In particular, for both $\lambda=(n-k,k) \vdash n$ and $\lambda=(n-k,1,\dots,1)\vdash n$, we provided a coarse bound based only on bounding the largest number of nonzero entries per column in $C(\mathcal{B}_{n},\mathcal{B}_{n-1}$). Despite the coarseness of these bounds, they greatly improve on the naive bound, which is $2(\dim(\mathbb{C}X^\lambda))^2$ computations (see Table \ref{table}).

In truth, both results could be refined further by bounding the number of nonzero entries in each column of each change of basis matrix, rather than focusing on the maximum possible number of nonzero entries in a column over all change of basis matrices $C(\mathcal{B}_j,\mathcal{B}_{j-1})$. Indeed, the proof of Theorem \ref{secondexthm} leads to a more refined bound than the theorem statement itself. The proof provides a bound for the number of nonzero entries in each individual column of $C(\mathcal{B}_n,\mathcal{B}_{n-1})$. Summing over the columns leads to a bound of at most 
\[
\sum_{\mu\in D(n-1,n-k-1)} (d_{\mu/(n-k-1)}+kd_{\mu/(n-k)})\dim(S^\mu)
\]
nonzero entries in $C(\mathcal{B}_n,\mathcal{B}_{n-1})$.  Continuing in this manner yields similar looking bounds for the number of nonzero entries in each change of basis matrix $C(\mathcal{B}_{j},\mathcal{B}_{j-1})$. Preliminary numerical results suggest that these more refined bounds are worth investigating. Indeed,  for small values of $n$ and $k$ they differ by only about a factor of 2 from the bound given in the recent algorithm of Clausen and H\"{u}hne \cite{clausen-and-huhne-2017}, which uses detailed knowledge about the appearance of repeated entries in the corresponding representation matrices. 

It is also worth noting that the results of Section \ref{Section:  General Upper Bounds} apply for any set of harmonic bases $\mathcal{B}_1,\dots,\mathcal{B}_n$ where $\mathcal{B}_1$ is the standard basis and $\mathcal{B}_j$ is an orbital symmetry adapted basis with respect to a  compatible collection $\mathcal{R}(G_1),\dots,\mathcal{R}(G_j)$ of irreducible representations. Is there a particular choice for these bases that can lead to a more efficient change of basis computation? Algorithms exist for constructing symmetry adapted bases (see, for example, \cite{fassler-and-stiefel-1992}) but at several key steps in these algorithms there is a degree of choice. As some choices could lead to more efficient bounds than others, it would be interesting to refine our bounds further by streamlining these basis constructions from the perspective of creating sparse change of basis matrices.


\bibliography{fast-algorithms-bib}{}
\bibliographystyle{plain}

\end{document}